\documentclass[reqno,11pt]{amsart}
\usepackage{a4wide,xcolor,eucal,enumerate,mathrsfs}
\usepackage[normalem]{ulem}
\usepackage{amsmath,amssymb,epsfig,amsthm} 
\usepackage[latin1]{inputenc}
\usepackage{psfrag}          

\numberwithin{equation}{section}

\newtheorem{theorem}{Theorem}[section]

\newtheorem{lemma}[theorem]{Lemma}
\newtheorem{proposition}[theorem]{Proposition}

\newtheorem{definition}[theorem]{Definition}

\newcommand{\N}{\mathbb{N}}

\newcommand{\be}{\begin{equation}}
\newcommand{\ee}{\end{equation}}

\newcommand{\R}{\mathbb{R}}

\renewcommand{\d}{{\mathrm d}}

\newcommand{\Sd}{\operatorname{S}^d}

\newcommand{\spt}{{\rm{spt}}}

\newcommand{\Id}{\operatorname{Id}}

\newcommand{\restr}[1]{\lower3pt\hbox{$|_{#1}$}}

\begin{document}

\title[Duality Theory for Multi-marginal OT in metric spaces]
{Duality Theory for Multi-marginal Optimal Transport with repulsive costs in metric spaces}

\author{Augusto Gerolin}
\author{Anna Kausamo}\author{Tapio Rajala}
\address{Department of Mathematics and Statistics \\
         P.O. Box 35 (MaD) \\
         FI-40014 University of Jyv\"askyl\"a \\
         Finland}
\email{augusto.gerolin@jyu.fi}		 
\email{anna.m.kausamo@jyu.fi}
\email{tapio.m.rajala@jyu.fi}


\thanks{The authors acknowledge the Academy of Finland projects no. 274372, 284511 and 312488.}
\subjclass[2000]{}
\date{\today}

\begin{abstract}
In this paper we extend the duality theory of the multi-marginal optimal transport problem for cost functions depending on a decreasing function of the distance (not necessarily bounded). This class of cost functions appears in the context of SCE Density Functional Theory introduced in \emph{Strong-interaction limit of density-functional theory} by M. Seidl \cite{Sei}.
\end{abstract}

\maketitle

\tableofcontents

\section{Introduction}
We consider the following multi-marginal optimal transport (MOT) problem 
\be\label{MKO}
\inf_{\gamma \in \Gamma(\rho)}\int_{X^N}c(x_1,\ldots,x_N)\,\d \gamma(x_1,\ldots, x_N),
\ee
where $(X,d)$ is a Polish space and  $\Gamma(\rho)$ denotes the set of Borel probability measures in $X^N$ having all $N$ marginals equal to a Borel probability measure $\rho$. We are interested in cost functions of the type
\[
c(x_1,\ldots, x_N)=\sum_{1\le i<j\le N}f(d(x_i,x_j)),
\]
where $f\colon]0,+\infty[\to \R$ is a continuous, decreasing function, not necessarily bounded above or below.
An interesting example of such cost is given by minus the logarithmic: $f(d(x,y)) = -\log(d(x,y))$. 

Our aim is to study properties of the so-called \textit{Kantorovich formulation} of \eqref{MKO} for such costs
\be \label{KMOT}
\sup \bigg\lbrace N\int_{X}u\,\d\rho ~\bigg|~ u\in L^1_\rho(X),\sum^N_{i=1}u(x_i)\le c(x_1,\ldots, x_N) \text{ for }\rho ^{\otimes (N)}\text{-a.e. } (x_1,\dots,x_N)\bigg\rbrace,
\ee
where $\rho ^{\otimes (N)}$ denotes the product of $N$ measures $\rho$. Optimal Transport problems with logarithmic-type costs were first considered in the literature by W. Wang \cite{Wang} and W. Gangbo and V. Oliker \cite{GanOli} motivated by the \textit{reflector problem}. In this case, $X = \Sd$, $N=2$ and the authors show the existence of optimal transport plans $\gamma = (\operatorname{Id},T)_{\sharp}\rho$ in \eqref{MKO} concentrated on the graph of a map $T\colon\Sd \to \Sd$. Generally, in the reflector problem, the marginals are not necessarily equal. 

In the multi-marginal case, logarithmic-type costs appear in Density Functional Theory (DFT), in the so-called \textit{strictly correlated limit} (SCE). In SCE-DFT, the multi-marginal optimal transport problem is interpreted as the equilibrium configuration of a distribution of $N$ charges in $(x_1,\dots,x_N) \in (\R^d)^N$ subject to the (minus) logarithmic electrostatic interaction depending on the distance between each two of the particles. Due to the indistinguishability of the particles, the charge density $\rho(x_i)$ is the same for all the particles $x_i, i=1,\dots,N$.

Although the interesting case in chemistry is when the system of $N$ electrons are in the physical space $X=\R^3$ subject to a Coulomb electronic-electronic interaction cost, in physics and mathematics $2$-body interactions other than the Coulombian one have been considered \cite{DMaGerNenGorSei, DMaGerNen, SeiGorSav, FMPCK, CorKarLanLee}, as well as the problem \eqref{MKO} in a lower space dimensions $X=\R^d, d=1,2$ \cite{CoDePDMa,Malet2012,ChenFriMendl, CoStra, LanDMaGerLeuGor}. In particular, when the particles are confined in the plane $\R^2$, the natural model of electrostatic potential between two charges $x_i$ and $x_j$ is given by the logarithmic interaction. We present in subsection  \ref{introwire} a pedagogical example of a charged wire, where the logarithmic electrostatic potential appears naturally. 

In the following, we give a brief overview on DFT-OT. For a complete presentation on the topic, we refer the reader to \cite{DMaGerNen} and the references therein.

\subsection{A brief review on the literature in DFT-OT}

The problem \eqref{MKO} when $X=\R^3$ and $c$ is the Coulomb cost $(f(\vert x-y\vert) = 1/\vert x-y\vert)$ was introduced in 1999 by M. Seidl \cite{Sei}. By using arguments from physics, Seidl suggested that, at least in the case when $\rho$ is radially symmetric, a minimizer $\gamma$ in \eqref{MKO} exists and is concentrated on the graph of a map $T:\R^3\to\R^3$, $T_{\sharp}\rho = \rho$, and its iterates, i.e.
\[
\gamma = (\operatorname{Id},T,T^{(2)},\dots,T^{(N-1)})_{\sharp}\rho,
\]
where $T^{(N)} = \operatorname{Id}$ and $T^{(i)}$ is the $i$-times composition of the map $T$ with itself. In particular, via the map $T$, the optimality condition in the Kantorovich formulation of \eqref{KMOT} with Coulomb cost reads
\begin{equation}\label{eq:gradkantmap}
\nabla u(x) = - \sum^N_{i=1} \dfrac{x -T^{(i)}(x)}{\vert x -T^{(i)}(x)\vert^3}.
\end{equation} 
As pointed out in \cite{Sei} (see also \cite{BuDePGor}), the constraint in \eqref{KMOT}, 
\[
\sum^{N}_{i=1}u(x_1) \leq \sum_{1 \leq i < j \leq N} \dfrac{1}{\vert x_i - x_j\vert},
\]
has a simple physical meaning: it is required that, at optimality, the allowed manifold of the full $3D$ configuration space is the minimum of the classical potential energy given by the Coulomb interaction. Also, the equation \eqref{eq:gradkantmap} means that 
if such an optimal map $T$ exists, the Kantorovich potential $u(x)$ must compensate the net force acting on the electron in $x$, resulting from the repulsion of the other $N -1$ electrons at positions $T^{(i)}(x)$ \cite{SeiGorSav}.

In Density Functional Theory (DFT), the problem \eqref{MKO} can be seen as a sort of a semi-classical limit (dilute limit of DFT) of the Hohenberg-Kohn functional\footnote{Also known as the Levy-Lieb functional.} \cite{HohKoh, Levy, Lieb83}. This was suggested in the physics literature by  
Gori-Giorgi, Seidl and Vignale \cite{GorSeiVig} and, proved rigorously in 2017 by Cotar, Friesecke and Kl\"uppelberg \cite{CFK,CFK17}. 

For the Coulomb cost in the $2$-marginal case $(N=2)$, the existence of a unique optimal transport plan  in \eqref{MKO} of type $\gamma = (\Id, T)_{\sharp}\rho$ $(N=2)$ was obtained, independently, by Cotar, Friesecke and Kl{\"u}ppelberg \cite{CFK} and by Buttazzo, De Pascale and Gori-Giorgi \cite{BuDePGor}. In the multi-marginal case $(N>2)$ on the real line $(d=1)$, Colombo, De Pascale and Di Marino \cite{CoDePDMa} proved the existence of optimal transport plans $\gamma = (\Id,T,\dots,T^{(N-1)})_{\sharp}\rho$ in \eqref{MKO} for Coulomb costs.  In \cite{DMaGerNenGorSei, DMaGerNen, SeiGorSav}, the repulsive harmonic cost$$ c_w(x_1,\dots,x_N) = -\sum_{1\leq i,j\leq N} \vert x_i - x_j\vert^2 $$
was studied: Friesecke \textit{et al} \cite{FMPCK} have shown the existence of optimal transport plans supported in $(N-1)d$-dimensional sets; in \cite{DMaGerNen} explicit examples of such higher dimensional optimal transport plans as well as an example of an optimal transport plan $\gamma$ concentrated on the graphs of $\Id, T,\dots, T^{(N-1)}$ for a nowhere continuous map $T\colon [0,1]^d\to [0,1]^d$ are presented.
In \cite{GerKauRajWeak}, we gave an example of a three-marginal harmonic repulsion case with absolutely continuous marginals in $\R^n$ for which 
there is a unique optimal transport plan which is not induced by a map.

\subsection{Logarithmic Eletrostatic potential: Charged wire}\label{introwire}
Consider a uniformly charged (infinitely thin) wire on the $z$-axis: 
\[
\mathcal{W}  := \lbrace \bold{x} = (x,y,z) \in \R^3 ~:~ \vert z \vert < \delta\rbrace, \qquad 0 < \delta \ll 1. 
\]
 Suppose that the wire has a charge density $\rho(\bold{x})$. The resulting electric field is defined by
\[
E(\bold{x}) = \dfrac{1}{4\pi \epsilon_0}\int_{\R^3}\dfrac{\bold{x}-s}{\vert \bold{x}-s\vert^3}\rho(s)\,\d s,
\]
where $\epsilon_0 > 0$ is a constant (permittivity of the free space). Due to \textit{Maxwell's first equation} (or \textit{Gauss' law} of eletrostatics) the scalar field $\rho\colon\R^3\to\R$ and the vector field $E(\bold{x})$ are related by
\[
\nabla\cdot E(\bold{x}) = \dfrac{1}{\epsilon_0}\rho(\bold{x}).
\]
We define the total amount of charge $Q_{\Omega}$ in a cylinder $\Omega = \Omega_{R,H} \subset \R^3$ of radius $R>0$ and height $H$, which has the wire as its axis of symmetry:
\be\label{introflux}
Q_{\Omega} = \int_{\Omega}\rho(s)\,\d s = \epsilon_0\int_{\Omega}\nabla\cdot E(x)\,\d x = \epsilon_0\oint_{\partial \Omega} E(a)\cdot \d a,
\ee
where the second equality is obtained using the Gauss' theorem. Due to symmetry, the magnitude $\vert E(\bold{x})\vert$ of the electric field depends only on the Euclidean distance $s = d(\bold{x},\bold{w}) = d(\bold{x},\mathcal{W})$ of a point $\bold{x}$ from the wire, $\vert E(\bold{x})\vert = E(s)$, i.e
$E(\bold{x}) = (E(s)\cos\theta,E(s)\sin\theta,0)$. Moreover, at each point $\bold{w}$ on the lateral surface of this cylinder, the vector $E(\bold{w})$ is normal to the surface and has everywhere the same magnitude $\vert E(\bold{w})\vert = E(R)$. 

Therefore, if $\rho(\bold{x}) = \overline{\rho} > 0$ is constant inside the cylinder, the flux integral and the total amount of charge in the cylinder $\Omega_{R,H}$ in \eqref{introflux} read 
\[
\dfrac{1}{\epsilon_0} \overline{\rho} H = (2\pi R)H\cdot E(R), \quad\text{and therefore,}\quad  E(R) = \dfrac{1}{2\pi\epsilon_0}\dfrac{1}{R}.
\]
Let us write $E(s) = 1/(2\pi\epsilon_0 s)$. Since $E(s) = - V'(s)$, the corresponding electrostatic potential $V(s)$ is of logarithmic from
\[
V(s) = -\dfrac{1}{2\pi\epsilon_0}\log\dfrac{s}{s_0}, \quad s_0 > 0. \bigskip
\]

\subsection{Kantorovich duality}

The duality \eqref{KMOT} and the existence of a maximizer in \eqref{KMOT} was shown by Kellerer \cite{Kel} in the case there exist $L_\rho^1(X)$-functions $h_1,\ldots,h_N$ and a constant $C$ such that 
\[C\le c(x_1,\ldots,x_N)\le h(x_1)+\cdots+h(x_N).\label{eq:conditions}\] 

More recently, De Pascale \cite{DeP} and Buttazzo, Champion and De Pascale  \cite{BuChaDeP} extended the duality theory 
for a class of  repulsive cost functions $c\colon\R^{dN}\to \R\cup\{+\infty\}$ which are bounded from below, allowing, for instance, the inclusion of the Coulomb ($s = 1$) and Riesz cost functions ($1 \leq s \leq d$)
\[
c(x_1,\dots,x_N) = \sum_{1\leq i<j\leq N} \dfrac{1}{\vert x_i - x_j\vert^{s}}.
\]

The main contribution of this paper is to extend the duality theory for logarithmic costs. Some of our proofs are based on arguments present in \cite{BuChaDeP}. One ingredient to tackle the problem of costs that are not bounded from below is to consider, for $R\in \, ]0,\infty[$, the truncated cost functions 
\be\label{eq:crtrunc}
c_R(x_1,\ldots, x_N):=\sum_{1\le i<j\le N}\max\{f(R),f(d(x_i,x_j))\},\text{ for all }(x_1,\ldots,x_N)\in X^N,
\ee
and related total cost $C_R$, and collection $\mathcal{F}_R$ of functions for the dual problem: 
\[
C_R(\gamma):=\int_{X^N}c_R(x_1,\ldots,x_N)\,\d \gamma(x_1,\ldots, x_N), \text{ for each }\gamma\in \Gamma(\rho),
\]
and
\[
\mathcal{F}_R:=\left\{u\in L^1_\rho(X)~\Big|~u(x_1)+\cdots +u(x_N)\le c_R(x_1,\ldots, x_N)\text{ for }\rho ^{\otimes (N)}\text{-a.e. }(x_1,\ldots, x_N)\right\}.
\]
In this paper, we will deal with the unbounded costs via the $\Gamma$-limit of their truncations.

\subsection{Organization of the paper}

This paper is divided as follows: in Section \ref{sec:gamma} we present the general setting and introduce briefly some properties of $\Gamma$-convergence. In Section \ref{sec:MKproblem}, we discuss the existence of a minimizer in \eqref{MKO} by assuming that the marginals $\rho$ satisfy, with respect to the function $f$ that appears in our cost $c$, a condition analogous to the common assumption of the marginal measures having finite second moments (see condition (B) in Section \ref{sec:MKproblem}). 

In Section \ref{sec:duality}, we extend the duality results of \cite{Kel, DeP, BuChaDeP} for a class of unbounded cost functions  (Theorem \ref{thm:dual}) and in Section \ref{sec:propKanto} we obtain regularity results of Kantorovich potentials (Theorem \ref{thm:ulip}) as well as continuity of the cost functional as a function of the marginal $\rho$. 

Finally, in Section \ref{sec:app} we give some applications of our results: we note the existence of optimal plans in \eqref{MKO}, for $\log$-type costs, which are concentrated on maps when $X=\R$, and we prove the existence of an optimal transport map for the logarithmic cost when $N=2$.

\section{Preliminaries}\label{sec:gamma}

\subsection{General assumptions}

Let $(X,d)$ a Polish space. We consider a Borel probability measure $\rho\in \mathcal{P}(X)$ having \textit{small concentration}, meaning 
\begin{equation}\label{hyp:rho}
\lim_{r\to 0}\sup_{x\in X}\rho (B(x,r))<\frac{1}{N(N-1)^2}.\tag{A}
\end{equation}

We denote by $(x_1,\ldots, x_N)$ points in $X^N$, so $x_i\in X$ for each $i$. 
If we do not otherwise specify,
each quantification with respect to $i$ or $i,j$ is from $1$ to $N$. 
For a fixed $N\ge 1$, we assume that the cost $c\colon X^N\to \R\cup\lbrace +\infty\rbrace$ is of the form
\begin{equation}\label{hyp:cost}
c(x_1,\ldots, x_N)=\sum_{1\le i<j\le N}f(d(x_i,x_j)),\quad \text{ for all } \, (x_1,\ldots,x_N)\in X^N,
\end{equation}
where $f\colon [0,\infty[\to\R\cup\{+\infty\}$ satisfies the following conditions

\begin{align*}\label{hyp:f}
&f|_{]0,\infty[}\text{ is continuous and decreasing, and}\tag{F1}\\
&\lim_{t\to 0+}f(t)=+\infty. \tag{F2}
\end{align*}

Let us denote for a fixed $R>0$, for all $t>0$
\begin{align*}
&f_R(t)=\begin{cases}
f(t)&\text{ if }t<R\\
f(R)&\text{ otherwise}\end{cases}~~~\text{ and }\\
&f_R^{-1}(t)=\inf\{s~|~f_R(s)=t\};\end{align*}
of course, if $f$ is not strictly decreasing, the inverse function $f^{-1}$ is not well defined, but still the \textit{left-inverse} of $f$ can be defined as above.

We denote the set of couplings or transport plans having $N$ marginals equal to $\rho$ by 
\[
 \Gamma(\rho)=\left\{\gamma\in \mathcal{P}(X^N)~\big|~\textrm{pr}^i_\sharp \gamma=\rho \text{ for all } i\right\},
 \]
where $\textrm{pr}^i$ is the projection on the $i$-th coordinate	
\[\textrm{pr}^i(x_1,\ldots,x_i,\ldots,x_N)=x_i, \quad \text{for all }(x_1,\ldots, x_i,\ldots ,x_N)\in X^N.\]
In addition, we set for each $\gamma\in\Gamma(\rho)$,
\[C(\gamma)=\int_{X^N}c(x_1,\ldots,x_N)\,\d \gamma(x_1,\ldots, x_N);\]
this is the transportation cost related to $\gamma$. 

We want to study the dual problem, so we set
\[
\mathcal{F}:=\left\{u\in L^1_\rho(X)~\Big|~u(x_1)+\cdots +u(x_N)\le c(x_1,\ldots, x_N)\text{ for }\rho ^{\otimes (N)}\text{-a.e. }(x_1,\ldots, x_N)\right\}
\]
and
\[
D\colon L^1_\rho(X)\to\R\cup \{-\infty,+\infty\},\quad D(u)=N\int_{X}u\,\d\rho\text{ for all }u\in L^1_\rho(X). 
\]

Here one should note that, in the definition of $\mathcal{F}$ and also in future considerations, we identify the elements of $L_\rho^1(X)$ with their representatives unless otherwise stated. That is why the constraint
\[u(x_1)+\cdots+u(x_N)\le c(x_1,\ldots,x_N)\]
is required to hold only for $\rho ^{\otimes (N)}$-almost-every $(x_1,\ldots, x_N)$. Also, we do not allow the representatives to get the value $+\infty$. This we may do without loss of generality, since $L^1$-functions are finite almost everywhere. 

We aim at showing that 
\be\min_{\gamma\in \Gamma(\rho)}C(\gamma)=\max_{u\in \mathcal{F}}D(u).\label{eq:duality}\ee

In order to guarantee the existence of a minimizer on the left-hand side of \eqref{eq:duality}, we also assume that there exist a point $o\in X$ and a radius $r_0>0$ such that
\[ \int_{X\setminus B(o,r_0)}f\left(2d(x,o)\right)\,\d\rho(x)> - \infty.\tag{B}\]

This is a similar assumption than requiring, in the case of quadratic cost, that the marginal measures have finite second moments. 

Notice that even when $X = \R^d$ the cost function $c$ in \eqref{hyp:cost} does not fall in the class of functions considered by Buttazzo, Champion and de Pascale \cite{BuChaDeP}, since it may not be bounded from below. However, by suitably truncating the cost $c$, the truncated functions $c_R$ are bounded from below for each $R$ and, modulo translation, fall into the category of functions considered in \cite{BuChaDeP}.

\subsection{$\Gamma$-convergence}
We briefly outline the relevant definitions and properties of $\Gamma$ and $\Gamma^+$-convergences. The former is a type of convergence of functionals adjusted to minimal value problems and the latter to maximal value problems.
For a thorough presentation of $\Gamma$-convergence, we refer the reader to Braides' book \cite{braidesbookgamma}.

\begin{definition}[$\Gamma$-convergence and $\Gamma^+$-convergence]
Let $(S,d)$ be a metric space. We say that a sequence $(F_n)_{n\in \N}$ of functions $F_n \colon S\to\overline\R$ $\Gamma$-converges to a function $F\colon S\to\R\cup \{-\infty,+\infty\}$ and denote $F_n\overset{\Gamma}{\to}F$ if for all $y\in S$ the following two conditions hold: 
\begin{align*}
&\text{For all sequences }(y_n)_{n\in \N}\text{ that converge to }y\text{ we have}\\
&\liminf_n F_n(y_n) \ge F(y),\text{ and}\tag{I}\\
&\text{there exists a sequence }(y_n)_{n\in \N}\text{ converging to }y\text{ such that}\\
&\limsup_n F_n(y_n)\le F(y).\tag{II}\end{align*}

Correspondingly, we say that a sequence $(D_n)_{n\in \N}$ of functions $D_n\colon S\to\overline\R$, $\Gamma^+$-convergence to a function $D\colon S\to\R\cup \{-\infty,+\infty\}$ and denote $D_n\overset{\Gamma^+}{\to}D$ if for all $u\in S$ the following two conditions hold: 
\begin{align*}
&\text{For any sequence }(u_n)_{n\in \N}\text{ converging to }u\text{ we have}\\
&\limsup_n D_n(u_n)\le D(u),\text{ and}\tag{I+}\\
&\text{there exists a sequence }(u_n)_{n\in \N}\text{ converging to }u\text{ such that }\\
&\limsup_n D_n(u_n)\le D(u).\tag{II+}\end{align*}
\end{definition}

In order to be able to take advantage of these notions, the underlying space $S$ must satisfy some compactness properties with respect to the minima/maxima of the functionals of interest. The following definition takes care of this. 
\begin{definition}
Let $(S,d)$ be a metric space. 
We say that a sequence $(F_n)_{n\in \N}$ of functions $F_n \colon S\to\R\cup \{-\infty,+\infty\}$ is equi-mildly coercive on $S$ if there exists a compact and non-empty subset $K$ of $S$ such that for all $n\in\N$ we have
\[\inf_{y\in S}F_n(x)=\inf_{y\in K} F_n(y).\]

Analogously, we say that a sequence $(D_n)_{n\in \N}$ of functions $D_n\colon S\to\R\cup \{-\infty,+\infty\}$ is equi-mildly ${}^+$-coercive on $S$ if there exists a compact and non-empty subset $K$ of $S$ such that for all $n\in\N$ we have
\[\sup_{u\in S}D_n(u)=\sup_{u\in K} D_n(u).\]
\end{definition}

\begin{theorem}{\cite[Theorem 1.21]{braidesbookgamma}}
\label{thm:existence}Let $(S,d)$ be a metric space. 
Let $(F_n)_{n\in \N}$ be an equi-mildly coercive sequence of functions $F_n \colon S\to \R\cup \{-\infty,+\infty\}$ that $\Gamma$-converges to some function $F\colon S\to\R\cup \{-\infty,+\infty\}$. Then there exists a minimum $y\in S$ of $F$ and the sequence $(\inf_{y\in S}F_n(y))_{n\in\N}$ converges to $\min_{y\in S}F(y)$. In addition, if $(y_n)_{n\in \N}$ is a sequence of elements of $S$ such that 
\[\lim_n F_n(y_n)=\lim_n \inf_{y\in S} F_n(y),\]
then every limit of a subsequence of $(y_n)_{n\in \N}$ is a minimizer of $F$. 

Similarly, let $(D_n)_{n\in \N}$ be an equi-mildly  $\Gamma^+$-coercive sequence of functions $D_n\colon S\to \R\cup \{-\infty,+\infty\}$ that $\Gamma^+$-converges to some function $D\colon S\to\R\cup \{-\infty,+\infty\}$. Then there exists a maximum $u\in S$ of $D$ and the sequence $(\sup_{u\in S}D_n(u))_n$ converges to $\max_{u\in S}D(u)$. In addition, if $(u_n)_{n\in \N}$ is a sequence of elements of $S$ such that 
\[\lim_n D_n(u_n)=\lim_n \sup_{u\in S} D_n(u),\]
then every limit of a subsequence of $(u_n)_{n\in \N}$ is a maximizer of $D$.

\end{theorem}

\section{Monge-Kantorovich problem}
\label{sec:MKproblem}

First, we prove the existence of a minimizer for the Monge-Kantorovich problem \eqref{MKO} in our framework.
Notice that the conditions (A) and (B) guarantee that the cost has a finite value.

\begin{proposition}\label{prop:existsminimizer}
Let $(X,d)$ be a Polish space. Suppose that $\rho\in \mathcal{P}(X)$ satisfies $(A)$ and $(B)$, and $c\colon X^N\to\R\cup\lbrace +\infty\rbrace$ is a cost function 
\[
c(x_1,\ldots, x_N)=\sum_{1\le i<j\le N}f(d(x_i,x_j)),\quad \text{ for all } \, (x_1,\ldots,x_N)\in X^N,
\]
where $f\colon [0,\infty[\to \R$ satisfies $(F1)$ and $(F2)$. Then, the following minimum is achieved
\[
\min_{\gamma\in \Gamma(\rho)}\int_{X^N}c(x_1,\dots,x_N)\,\d\gamma(x_1,\dots,x_N).
\]
\end{proposition}
\begin{proof}
 The proof follows standard arguments. From \cite{Kel} we know that $\Gamma(\rho)$ is compact. 
 Therefore, it suffices to prove the lower semicontinuity of the cost $C(\gamma)$.
 For this, it suffices (see \cite[Theorem 4.3]{Vil03}) to find an upper semicontinuous function $h$ such that
 \begin{align}
&h\in L^1_\gamma(X^{N})\text{ for all }\gamma \in \Gamma(\rho),\label{eq:integrable}\\
&c\ge h \text{, and }\label{eq:lowerbound}\\
&\int_{X^{N}}h\,\d\gamma'=\int_{X^{N}}h\,\d\gamma\, \text{ for all }\gamma,\gamma' \in \Gamma(\rho).\label{eq:integralsconverge}
\end{align}
Let us define $g\colon [0,\infty[\to\R$ by
\[g(r)=\begin{cases}
f(r_0)&\text{ if }r<r_0\\
f(r)&\text{ if }r\ge r_0\end{cases},\]
and set $h\colon X^{N}\to\R$
\[h(x_1,\ldots, x_N)=\frac 12\sum_{1\le i<j\le N}(g(2d(x_i,o))+g(2d(x_j,o))).\]
As a finite sum of continuous functions, $h$ is continuous and thus trivially upper semicontinuous. In addition, for any $\gamma\in \Gamma(\rho)$ we have
\begin{align*}
\int_{X^N}h\,\d\gamma&=\frac 12\sum_{1\le i<j \le N}\int_{X^N}(g(2d(x_i,o))+g(2d(x_j,o))\,\d\gamma \\
&=\frac 12N(N-1)\int_{X}g(2d(x_i,o))\,\d\rho(x)\nonumber\\
&=\frac 12N(N-1)\left(\int_{B(o,r_0)}f(2\cdot 2r_0)\,\d\rho(x)+\int_{X\setminus B(o,r_0)}f(2d(x,o))\,\d\rho(x)\right).
\end{align*}
Therefore, due to Assumption (B) condition \eqref{eq:integrable} holds. Similarly, condition \eqref{eq:integralsconverge} follows by
\[\int_{X^N}h\,\d\gamma'=\frac12 \sum_{1\le i<j\le N}\int_{X}(g(2d(x_i,o))+g(2d(x_j,o)))\,\d\rho=\int_{X^N}h\,\d\gamma.\]
Finally, to prove condition \eqref{eq:lowerbound}, we fix $(x_1,\ldots, x_N)\in X^N$ and by (F1) we have that
\begin{align*}
c(x_1,\ldots, x_N)&=\sum_{1\le i<j\le N}f(d(x_i,x_j))\ge\sum_{1\le i<j\le N}g(d(x_i,x_j))\nonumber\\
&\ge\sum_{1\le i<j\le N}g(d(x_i,o)+d(x_j,o))\nonumber\\
&\ge\sum_{1\le i<j\le N}g(2\max\{d(x_i,o),d(x_j,o)\})\nonumber\\
&=\sum_{1\le i<j\le N}\min\{g(2d(x_i,o)),g(2d(x_j,o))\}\nonumber\\
&=\frac 12\sum_{1\le i<j\le N}(g(2d(x_i,o))+g(2d(x_j,o))-|g(2d(x_i,o))-g(2d(x_j,o))|)\nonumber\\
&\ge\frac 12\sum_{1\le i<j\le N}(g(2d(x_i,o))+g(2d(x_j,o))-0)=h(x_1,\ldots, x_N).
\end{align*}
This concludes the proof.
\end{proof}

For $\alpha > 0$ we define the set $D_{\alpha}$ as  
\[
D_\alpha:=\left\{(x_1,\ldots, x_N)\in X^N~|~\text{there exist }i,j\text{ such that }d(x_i,x_j)<\alpha\right\}.
\]

The next theorem states that for any measure $\rho$ there exists $\overline{\alpha}>0$ for which the support of any optimal plan is concentrated away from the set $D_{\overline{\alpha}}$. 

\begin{theorem}\label{thm:offdiagonal}
Let $(X,d)$, $\rho$, $f$, $c$ as in the Proposition \ref{prop:existsminimizer} and let $\gamma$ be a minimizer of
\[
C(\rho) =  \min_{\gamma\in\Gamma(\rho)}\int_{X^N} c(x_1,\dots,x_N)\,\d\gamma(x_1,\dots,x_N).
\]
Let us fix $0<\beta<1$ such that
\[\sup_{x\in X}\rho (B(x,\beta)) <\frac{1}{N(N-1)^2}.\]
Then, we have for all 
\be\label{eq:alphaneedstobe}
\alpha<f^{-1}\left(\frac{N^2(N-1)}{2}f(\beta)\right)\ee
the inclusion
\be\label{eq:claim}
\spt (\gamma)\subset X^N\setminus D_\alpha.\ee
\end{theorem}
\begin{proof}
The proof presented in \cite{BuChaDeP} also works here. The fact that optimal plans stay out of the diagonal reflect the properties of the cost close to the singularity, not to the tail. 
\end{proof}

We recall that for all $R>0$, the truncated costs $c_R$ and $C_R$
\[c_R(x_1,\ldots, x_N)=\sum_{1\le i<j\le N}\max\{f(R),f(d(x_i,x_j))\}\text{ for all }(x_1,\ldots,x_N)\in X^N,\]
\[
C_R(\gamma)=\int_{X^N}c_R(x_1,\ldots,x_N)\,\d \gamma(x_1,\ldots, x_N)\text{ for each }\gamma\in \Gamma(\rho).
\]
Using these we define the functionals $K_R,K\colon\mathcal{P}(X^N)\to\R\cup \{+\infty\}$, 
\[K_R(\gamma):=\begin{cases}C_R(\gamma)&\text{ if }\gamma\in \Gamma(\rho)\\
+\infty&\text{ otherwise}\end{cases},\]
\[K(\gamma):=\begin{cases}C(\gamma)&\text{ if }\gamma\in \Gamma(\rho)\\
+\infty&\text{ otherwise}\end{cases}. \]

An approximation result of convergence of minimizers of the truncated costs $(K_R)_{R\in\N}$ is given by the following proposition.

\begin{proposition}\label{prop:gamma}The sequence of functionals $(K_R)_{R\in \N}$ is equicoercive and $\Gamma$-converges to $K$ with respect to the weak convergence of measures. 
\end{proposition}
\begin{proof}
First we notice that the equicoerciviness of $(K_R)_{R\in \N}$ follows from the fact that $\Gamma(\rho)$ is weakly compact \cite{Kel}.
We then fix $\gamma\in \mathcal{P}(X^N)$ and show that
\begin{align}
&\text{for all sequences }(\gamma_R)_{R\in\N}\text{ such that }\gamma_R\rightharpoonup \gamma\text{ we have}\nonumber\\
&\liminf_{R\to\infty} K_R(\gamma_R)\ge K(\gamma)\text{, and }\label{eq:inferior}\\
&\text{there exists a sequence }(\gamma_R)_{R\in\N}\text{ such that }\gamma_R\rightharpoonup\gamma\text{ and }\nonumber\\
&\limsup_{R\to\infty} K_R(\gamma_R)\le K(\gamma).\label{eq:superior}\end{align}

Fix a sequence $(\gamma_R)_{R\in\N}$ in $\mathcal{P}(X^N)$ such that $\gamma_R\rightharpoonup  \gamma$. 
By going to a subsequence we may assume that $\liminf_{R\to\infty} K_R(\gamma_R) = \lim_{R\to\infty}  K_R(\gamma_R)$.
Thus, we may also suppose that $K_R(\gamma_R)<\infty$ for all $R \in \N$, since otherwise \eqref{eq:inferior} would trivially hold.
Consequently, we have that $\gamma_R \in \Gamma(\rho)$ for all $R \in \N$ and thus also $\gamma \in \Gamma(\rho)$ by compactness of $\Gamma(\rho)$, see \cite{Kel}.
Now, by monotonicity of the integral and lower semi-continuity of $K(\gamma)$ we get
\[
\liminf_{R\to\infty} K_R(\gamma_R) \ge \liminf_{R\to\infty} K(\gamma_R)\ge K(\gamma),
\]
so \eqref{eq:inferior} is satisfied. Finally, the condition \eqref{eq:superior} is satisfied by the constant sequence $\gamma_R = \gamma$ for all $R \in \N$. 
\end{proof}

\subsection{Symmetric probability measures}

We remark that the Monge-Kantorovich problem \eqref{MKO} can be restricted to symmetric transport plans.

\begin{definition}[Symmetric measures]\label{def:symmeasures}
A measure $\gamma \in \mathcal{P}(X^N)$ is symmetric if 
\[
\int_{X^N}\phi(x_1,\dots,x_N)\,\d\gamma = \int_{X^N} \phi(\overline{\sigma}(x_1,\dots,x_N))\,\d\gamma, \text{ for all } \phi \in  \mathcal{C}(X^N)
\]
and for all permutations $\overline{\sigma}$ of $N$ symbols. We denote by $\Gamma^{sym}(\rho)$, the space of all $\gamma \in \Gamma(\rho)$ which are symmetric.
\end{definition}

\begin{proposition}\label{prop:symplans}
Let $(X,d)$ be a Polish space. Suppose $\rho\in \mathcal{P}(X)$ such that $(A)$ and $(B)$ hold and $c\colon X^N\to\R\cup\lbrace +\infty\rbrace$ is a continuous cost function. Then,
\begin{equation}\label{eq:MKsymMK}
\min_{ \gamma \in \Gamma(\rho) } \int_{X^{N} } c(x_1,\ldots, x_N) \, \d  \gamma = \min_{ \gamma \in \Gamma^{sym}(\rho) } \int_{X^{N} } c(x_1,\ldots, x_N) \, \d  \gamma. 
\end{equation}
\end{proposition}

\begin{proof}
The minimum on the left-hand side in \eqref{eq:MKsymMK} is surely smaller than or equal to the minimum on the right-hand side, since $\Gamma^{sym}(\rho) \subset \Gamma(\rho)$. Suppose $\gamma \in \Gamma(\rho)$, we can define a symmetric plan 
\[
\gamma_{sym} = \dfrac{1}{N!}\sum_{\sigma \in \mathfrak{S}_N}\sigma_{\sharp}\gamma, \quad \sigma \in \mathfrak{S}_N,
\]
where $\mathfrak{S}_N$ is the set of permutation of $N$-symbols. Thanks to the linearity of the cost function $C(\gamma)$, $\gamma_{sym}$ and $\gamma$ have the same cost and, therefore, \eqref{eq:MKsymMK} holds.
\end{proof}

\section{Duality Theory for log-type cost functions}
\label{sec:duality}

The following theorem extends Kantorovich duality for our class of cost functions. 

\begin{theorem}\label{thm:dual}
Let $(X,d)$ be a Polish space. Suppose $\rho\in \mathcal{P}(X)$ such that $(A)$ and $(B)$ hold and $c\colon X^N\to\R\cup\lbrace +\infty\rbrace$ is a cost function 
\[
c(x_1,\ldots, x_N)=\sum_{1\le i<j\le N}f(d(x_i,x_j)),\quad \text{ for all } \, (x_1,\ldots,x_N)\in X^N,
\]
where $f\colon [0,+\infty[\to \R\cup\lbrace +\infty\rbrace$ is a function satisfying $(F1)$ and $(F2)$. Then, the duality holds: 
\begin{equation}\label{eq:dualityagain}
\min_{\gamma\in\Gamma(\rho)}\int_{X^N}c\,\d\gamma = \max_{u\in L^1_\rho(X)}\bigg\lbrace N\int_X u(x)\,\d\rho(x) ~:~ \sum^N_{i=1}u(x_i)\le c(x_1,\ldots, x_N)\text{ }\rho ^{\otimes (N)}\text{-a.e.} \bigg\rbrace.
\end{equation}
\end{theorem}

\begin{proof}
Due to Proposition \ref{prop:existsminimizer} the minimum on the left-hand side is realized. 
By using the monotonicity of integral and the fact that $\gamma \in \Gamma(\rho)$, we easily get 
\[
\min_{\gamma\in \Gamma(\rho)}C(\gamma)\ge \sup_{u\in \mathcal{F}} D(u).
\] 
Hence, we need to show that
\be\label{eq:newclaim}
\min_{\gamma\in \Gamma(\gamma)}C(\gamma)\le \sup_{u\in \mathcal{F}} D(u)
\ee
and that a maximizer for $\max_{u\in \mathcal{F}} D(u)$ exists.

Towards this goal, let us fix a minimizer $\gamma$ of $C$. It now suffices to show that there exists a function $u\in \mathcal{F}$ such that
\[C(\gamma)\le D(u)\,.\]

For each $L>0$, let us denote $\gamma_L = \gamma\restr{B(o,L)^N}$, and by $\gamma_L^P$ the normalized versions of $\gamma_L$. Notice that
because of Assumption (B), $\gamma_L \ne 0$ for large enough $L>0$. 
Let us denote the marginals of $\gamma_L^P$ by  $\rho_L$.

Now, $\gamma_L^P$ is optimal also for all $C_R$ with $R\ge 2L$, since $C = C_R$ for all couplings of $\rho_L$. 
Let $(u_R)$ be a sequence of Kantorovich potentials, each corresponding to $\gamma_{R/2}^P$ with the cost $c_R$ and the marginals $\rho_{R/2}$.
By \cite[Lemma 3.3]{BuChaDeP}, we may assume that for all $R$ and all $x_1\in X$ we have the representation 
\begin{equation}\label{eq:presentationR}
u_R(x_1)=\inf\left\{\sum_{i=1}^N c_R(x_1,x_2,\ldots,x_N)-\sum_{j=2}^Nu_R(x_j)~\bigg|~(x_2,\ldots,x_N)\in X^{N-1}\right\}.
\end{equation}
Let us fix $R_0>0$ such that $\gamma_{R_0/2} \ne 0$,
and a point $(\overline x_1, \dots, \overline x_N) \in \spt(\gamma_{R_0/2})$.

We may then assume that for all $R \ge R_0$, we have
\[
u_R(\overline x_i)=\frac 1Nc_R(\overline x_1,\ldots, \overline x_N)=\frac 1Nc(\overline x_1,\ldots, \overline x_N)\text{ for all }i,
\]
since $(\overline x_1, \dots, \overline x_N) \in \spt(\gamma_{R_0/2}) \subset \spt(\gamma_{R/2})$.

Now we have, for all $R \ge R_0$ and for all $x=(x_1,\ldots,x_N)\in X^N$, by \eqref{eq:presentationR} and Theorem \ref{thm:offdiagonal}, for some $\alpha>0$ the estimate 
\begin{align*}
u_R(x_1)&\le c_R(x_1,\overline x_2, \dots, \overline x_N) - \frac {N-1}Nc_R(\overline x_1,\ldots, \overline x_N)\\
& \le \frac{N(N-1)}{2}f(\frac{\alpha}2)- \frac {N-1}Nc_R(\overline x_1,\ldots, \overline x_N) =: M,
\end{align*}
since by the fact that $(\overline x_1, \dots, \overline x_2) \in X^N\setminus D_\alpha$, we may assume (by changing $\overline x_1$ with some other $\overline x_i$), that $d(x_1,\overline x_j) \ge \frac{\alpha}{2}$ for all $j \in \{2,\dots, N\}$.

For the lower bound, we use again the representation \eqref{eq:presentationR} and the upper bound that we just obtained. For all $x=(x_1,\ldots,x_N)\in \spt(\gamma_{L}^P)$, when $R \ge 2L$, we have 
\begin{align*}
u_R(x_1)&=\sum_{1\le i<j\le N}f(d(x_i,x_j)) - \sum_{j=2}^Nu_R(x_j) \\
& \ge \frac{N(N-1)}{2}f(2L) - (N-1)M.
\end{align*}

What we have shown is that for each $L$ the sequence $(u_R)$ is bounded on $\spt\rho_L$ when $R \ge 2L$. So, we may in each set $\spt(\rho_L)$ define $u$ as the weak limit of $u_R$ along some subsequence, and finally define $u$ in the whole space by a diagonal argument.
Now, assuming that we have that $u \in \mathcal F$, by the definition of $\gamma_L$, and by the weak convergence we get
\[
 C(\gamma) = \lim_{R \to \infty}C(\gamma_{R/2}^P) = \lim_{R \to \infty}D_R(u_R) = D(u).
\]
Thus, it remains to show that $u \in \mathcal F$. Supposing this is not the case, 
there exists a Borel set $A\subseteq X^N$ such that $\rho ^{\otimes (N)}(A)>0$ and 
\begin{equation}\label{eq:boundfails}
u(x_1)+\cdots +u(x_N)>c(x_1,\ldots, x_N)\text{ for all }(x_1,\ldots, x_N)\in A.
\end{equation}
By going into a subset of $A$ if necessary, we may assume that $A \subset (\spt\rho_L \cap B(0,L))^N$ for some $L>0$.
Now, by Mazur's lemma, there is a sequence $(\tilde u_R)$ of convex combinations of $(u_R)_{R \ge 2L}$ strongly converging to $u$ in $L^1(\rho)$.
Since, $c_R = c$ on $A$ for all $R \ge 2L$, we have
\begin{equation}\label{eq:boundholds}
\tilde u_R(x_1)+\cdots +\tilde u_R(x_N) \le c(x_1,\ldots, x_N)\text{ for all }(x_1,\ldots, x_N)\in A,
\end{equation}
for all $R \ge 2L$, as the inequality is preserved under convex combinations.

Let us denote
\[l:=\int_{A}(u(x_1)+\cdots+u(x_N)-c(x_1,\ldots,x_N))\,\d\rho^{\otimes (N)}.\]
Due to \eqref{eq:boundfails} we have $l>0$. 
Because $\tilde u_R \to u$ strongly, there exists $R_1 \ge 2L$ such that 
\be
\int_{A}\sum_{i=1}^N |\tilde u_R(x_i)-u(x_i)|\,\d\rho ^{\otimes (N)}<\frac l2\text{ for all }R\ge R_1.\label{eq:uestimate}
\ee
Then we have for all $R>R_1$
\begin{align*}
\int_{A}&\left(\sum_{i=1}^N\tilde u_R(x_i)-c(x_1,\ldots,x_N)\right)\,\d\rho ^{\otimes (N)}\\
&=\int_A\sum_{i=1}^N(\tilde u_R(x_i)-u(x_i))\,\d\rho ^{\otimes (N)}+\int_A\sum_{i=1}^Nu(x_i)-c(x_1,\ldots, x_N)\,\d\rho ^{\otimes (N)}\\
& >l-\frac l2=\frac l2>0,\end{align*}
contradicting \eqref{eq:boundholds}.
\end{proof}

\section{Properties of the Kantorovich potentials}
\label{sec:propKanto}

Let $C(\gamma)$ be as before
\[
C(\gamma) = \int_X \sum_{1\le i<j\le N}f(d(x_i,x_j)) \,\d\gamma.
\]
We denote by $C^R(\gamma)$ the truncation of a cost $C(\gamma)$ from above\footnote{Notice that we have used the notation $C_R$ to correspond to the cost truncated from below.}, 
\[C^{R}(\gamma)=\int_{X^N}c^R(x_1,\dots,x_N)\,\d\gamma, \, \text{ for all }\gamma \in \mathcal{P}(X^N),\] 
where we have denoted by $c^R$ the corresponding truncation of $c$,
\[
 c^R(x_1,\dots,x_N) = \sum_{1\le i<j\le N}\min\{R,f(d(x_i,x_j)\}).
\]

\begin{proposition}\label{prop:truncatedequality}
Let $\rho\in \mathcal{P}(X)$ satisfy the assumptions $(A)$ and $(B)$. Fix $\beta>0$ such that 
\[\sup_{x\in X}\rho(B(x,\beta))<\frac{1}{N(N-1)^2}.\]
Then,  for any $\alpha<f^{-1}\left(\frac{N^2(N-1)}{2}f(\beta)\right)$ and for all optimal $\gamma \in \Gamma(\rho)$ associated to $C(\gamma)$, we have
\be\label{eq:firstclaim}
C(\gamma)\le \frac{N^3(N-1)^2}{4}f(\beta)\quad \text{ and } \quad C(\gamma) = C^{f(\alpha)}(\gamma).\ee
Moreover, for the same $\alpha$, any Kantorovich potential $u_\alpha$ for $C^{f(\alpha)}$ is also a Kantorovich potential for $C$.\label{eq:thirdclaim}
\end{proposition}
\begin{proof}
For each 
\[\alpha<f^{-1}\left(\frac{N^2(N-1)}{2}f(\beta)\right),\]
we know by Theorem \ref{thm:offdiagonal} that the support of $\gamma$ can intersect at most the boundary of $D_\alpha$. Therefore, since $f$ is 
decreasing, we have for all $(x_1,\ldots, x_N)\in \spt(\gamma)$ the estimate 
\[c(x_1,\ldots, x_N)= \sum_{1\le i<j\le N}f(d(x_i,x_j))\le \frac{N(N-1)}{2}f(\alpha).\]
Thus, since $\gamma$ is a probability measure, we have
\[C(\gamma)\le \int_{X^N}\frac{N(N-1)}{2}f(\alpha)\,\d\gamma=\frac{N(N-1)}{2}f(\alpha).\]
Taking $\alpha \to f^{-1}\left(\frac{N^2(N-1)}{2}f(\beta)\right)$, we then get
\begin{align*}
 C(\gamma)& \le \frac{N(N-1)}{2}f\left(f^{-1}\left(\frac{N^2(N-1)}{2}f(\beta)\right)\right)\\
 & =\frac{N(N-1)}{2}\cdot\frac{N^2(N-1)}{2}f(\beta)=\frac{N^3(N-1)^2}{4}f(\beta),
\end{align*}
which gives the left-hand side in \eqref{eq:firstclaim}. Let us then fix an optimal plan $\gamma_\alpha$ for the cost $C^{f(\alpha)}$. Then $\spt(\gamma_\alpha)\in X^N\setminus D_\alpha$, so $c=c^{f(\alpha)}$ on $\spt(\gamma_\alpha)$. Thus, 
\[C(\gamma)\le \int_{X^N}c\,\d\gamma_\alpha=\int_{X^N}c^{f(\alpha)}\,\d\gamma_\alpha=C^{f(\alpha)}(\gamma_\alpha).\]
The opposite inequality is simply due to the monotonicity of the integral. It remains to prove the last part of the statement. We fix a Kantorovich potential $u_\alpha$ for $C^{f(\alpha)}$. It satisfies, for $\rho ^{\otimes (N)}$-almost every $(x_1,\ldots, x_N) \in X^N$ the estimate
\[u_\alpha(x_1)+\cdots +u_\alpha(x_N)\le c^{f(\alpha)}(x_1,\ldots,x_N)\le c(x_1,\ldots, x_N).\]
Hence,  $u_\alpha$ is also a Kantorovich potential for the cost function $c$ and, moreover,
\[\int_X u(x)\,\d\rho(x) = \min_{\gamma\in \Gamma(\rho)} C(\gamma) = \min_{\gamma\in \Gamma(\rho)} C^{f(\alpha)}(\gamma) = N\int_X u_{\alpha}(x)\,\d\rho(x).\]
This concludes the proof.
\end{proof}

\begin{theorem}\label{thm:ulip} Let $(X,d)$ be a Polish space. Suppose $\rho\in \mathcal{P}(X)$ such that $(A)$ and $(B)$ hold and $c\colon X^N\to\R\cup\lbrace +\infty\rbrace$ is a cost function 
\[
c(x_1,\ldots, x_N)=\sum_{1\le i<j\le N}f(d(x_i,x_j)),\quad \text{ for all } \, (x_1,\ldots,x_N)\in X^N,
\]
where $f\colon [0,+\infty[\to \R\cup\lbrace +\infty\rbrace$ is a function satisfying $(F1)$ and $(F2)$.

Let $\beta>0$ be such that
\[\sup_{x\in X}\rho (B(x,\beta)) <\frac{1}{N(N-1)^2}.\]
 Assume additionally that, for some $\alpha<f^{-1}\left(\frac{N^2(N-1)}{2}f(\beta)\right)$, the restriction $f|_{[\alpha,\infty[}$ is Lipschitz. Then,  there exists a Kantorovich potential $w$ in \eqref{eq:dualityagain}  that is Lipschitz.
\end{theorem}

The following lemma is useful for proving Theorem \ref{thm:ulip}. The proof follows in the same way as the proof of \cite[Lemma 3.3]{BuChaDeP}.

\begin{lemma}\label{lm:infpotentiallemma}
Let $u$ be a Kantorovich potential for the problem \eqref{KMOT}, i.e. a maximizer of the problem \eqref{KMOT}.  Then there exists a Kantorovich potential $\tilde u$ such that $\tilde u\ge u$ which satisfies the representation 
\be\label{eq:infpotential}
\tilde u(x)=\inf\left\{c(x,x_2,\ldots, x_N)-\sum_{i\ge 2}\tilde u(x_i)~:~x_j\in X\text{ for all }j\right\}\,.\ee
\end{lemma}

\begin{proof}[Proof of the Theorem \ref{thm:ulip}] 
According to Lemma \ref{lm:infpotentiallemma}, we may choose a Kantorovich potential $u_\alpha$ for the truncated cost $C^{f(\alpha)}$ satisfying, for all $x\in X$,
\[u_\alpha(x)=\inf\left\{c^{f(\alpha)}(x,x_2,\ldots, x_N)-\sum_{j=1}^N u_\alpha(x_j)~|~x_j\in X\right\}.\]
By Proposition \ref{prop:truncatedequality}, due to the choice of $\alpha$, $u_\alpha$ is also a Kantorovich potential for $C$. So, it suffices to show that $u_\alpha$ is Lipschitz. Since $f|_{[\alpha,\infty[}$ and $d$ are Lipschitz, the function $h\colon X\to\R\cup \{-\infty,+\infty\}$, 
\[h(x)=\sum_{1\le i<j\le N}c^{f(\alpha)}(x,x_2,\ldots, x_N)-\sum_{j=2}^Nu_\alpha(x_j)~~~\text{for all }x\in X,\]
is Lipschitz with a Lipschitz constant that does not depend on $(x_2,\ldots, x_N)$. Since the infimum of a family of uniformly Lipschitz functions is Lipschitz, 
we have that $u_{\alpha}$ is Lipschitz.
\end{proof}

Finally, we can move on to the continuity properties of the cost functional $C(\rho)$ with respect to the marginal $\rho$. 

\begin{proposition}\label{thm:ccontinuous}
Under the same assumptions as in Theorem \ref{thm:ulip}, let $(\rho_n)$ be a sequence in $\mathcal{P}(X^N)$, weakly converging to some $\rho_\infty\in \mathcal{P}(X^N)$ that satisfies $(A)$. If
\begin{equation}\label{eq:unifconvergence}
\int_{X\setminus B(o,r)}f\left(2d(x,o)\right)\,\d\rho_n(x) \to 0 \quad \text{ uniformly when } r\to 0,
\end{equation}
then 
\[
\lim_{n\to \infty}C(\rho_n)=C(\rho).
\]
\end{proposition}
\begin{proof}
By \cite[Theorem 3.9]{BuChaDeP}, the above result holds for the singular costs $C_R$ which are bounded from below.
Therefore, it suffices to show that for each $\varepsilon > 0$ there exists $R \in \N$ such that
\[
 |C(\rho_n) - C_R(\rho_n)| < \varepsilon
\]
for all $n \in \N \cup \{\infty\}$.
Since the inequality $C \le C_R$ always holds, it suffices to show that $C_R(\rho_n) - C(\rho_n) < \varepsilon$ for $R$ large enough.
In order to obtain this, we take a minimizer $\gamma_n$ for $C$ with marginals $\rho_n$ (given by Proposition \ref{prop:existsminimizer}) and estimate, assuming $f(R/2) \le 0$ by taking $R$ large enough and $\gamma_n \in \Gamma^{sym}(\rho_n)$ by Proposition \ref{prop:symplans},
\begin{align*}
 C_R(\rho_n)- C(\rho_n) & \le \int_{X^N} (C_R-C)\,\d\gamma_n = \int_{X^N} \sum_{1\le i<j\le N}\max\{f(R) - f(d(x_i,x_j)), 0\}\,\d\gamma_n\\
 & \le - N(N-1)\int_{d(x_1,x_2) \ge R} f(d(x_1,x_2))\,\d\gamma_n\\
 & \le - N(N-1)\int_{d(x_1,x_2) \ge R} f(\max\{2d(x_1,o),2d(x_2,o)\})\,\d\gamma_n\\
 & \le - 2N(N-1)\int_{d(x,o) \ge \frac{R}2} f(2d(x,o))\,\d\rho_n < \varepsilon,
\end{align*}
for large enough $R$ by assumption \eqref{eq:unifconvergence}.
\end{proof}

\section{Monge Problem for $\log$-type costs}
\label{sec:app}

Regarding the existence of Monge-type minimizers in \eqref{MKO}, the first positive result for repulsive type costs is shown in \cite{CoDePDMa} where, in dimension $d=1$, $X=\R$, M. Colombo, L. De Pascale and S. Di Marino prove that, for an absolutely continuous measure, a symmetric optimal plan $\gamma$ is always induced by a cyclical optimal map $T$. One important ingredient of that proof relied on the fact that for symmetric cost functions \eqref{MKO} can be restricted for a class of symmetric transport plans (see Definition \ref{def:symmeasures} and Propostion \ref{prop:symplans}).

\begin{theorem}[Colombo, De Pascale and Di Marino, \cite{CoDePDMa}]\label{teo:1DN} Let $\mu \in \mathcal{P}(\R)$ be an absolutely continuous probability measure and $f\colon\R\to \R$ strictly convex, bounded from below and non-increasing function. Then there exists a unique optimal symmetric plan $\gamma \in \Gamma^{sym} (\mu)$ that solves
$$ \min_{ \gamma \in \Gamma^{sym}(\mu) } \int_{\R^{N} } \sum_{1 \leq i < j \leq N}  f(|x_j-x_i|) \, \d  \gamma. $$
Moreover, this plan is induced by an optimal cyclical map $T$, that is, $\gamma_{sym}=\frac 1{N!} \sum_{\sigma \in \mathfrak{S}_N} \sigma_{\sharp} \gamma_T$, where $\gamma_T=(Id,T,T^{(2)} , \ldots, T^{(N-1)})_{\sharp} \mu$. An explicit optimal cyclical map is  
$$ T(x) =\begin{cases}  F_{\mu}^{-1} (F_{\mu}(x) + 1/N) \qquad & \text{ if }F_{\mu}(x) \leq (N-1)/N \\ F_{\mu}^{-1} ( F_{\mu}(x) +1 - 1/N ) & \text{ otherwise.} \end{cases}$$
Here $F_{\mu}(x)=\mu ( -\infty , x]$ is the distribution function of $\mu$, and $F_{\mu}^{-1}$ is its lower semicontinuous left inverse. 
\end{theorem}

We remark that, due to Theorem \ref{thm:existence}, the above Theorem \ref{teo:1DN} also holds for unbounded cost functions satisfying $(F1)$ and $(F2)$ and under the additional assumption $(B)$ on the absolutely continuous measure $\mu$. This can be seen for instance by taking a minimizer for the unbounded cost and observing that its restriction to a bounded set is also a minimizer of a truncated for and thus of the form given by Theorem \ref{teo:1DN}.

\subsection{$\operatorname{Log}$-type cost ($N=2$)}
Here we consider $X = \R^d$ with $d \ge 1$.

\begin{theorem}\label{thm:N2Monge}Let $\rho \in \mathcal{P}(\R^d)$ be a probability measure such that $(A)$ and $(B)$ hold. Then there exists a unique optimal plan $\gamma_O \in \Gamma (\rho,\rho)$ for the problem
\be\label{app:n2}
\min_{ \gamma \in \Gamma(\rho,\rho) } \int_{\R^{d}\times \R^d } -\log(|x_1-x_2|) \, \d  \gamma(x_1,x_2). 
\ee
Moreover, this plan is induced by an optimal map $T$, that is, $\gamma= (Id,T)_{\sharp} \rho$, and $T(x)=x -\frac{ \nabla u } {| \nabla u|^2 }$ $\rho$-almost everywhere, where $u$ is a Lipschitz maximizer for the dual problem \eqref{KMOT}.
\end{theorem}

\begin{proof} Let us consider $\gamma$ a minimizer for the problem \eqref{app:n2} and $u$ a maximizer of the dual problem, which is Lipschitz by Theorem \ref{thm:ulip}. Then,  
$$ F(x_1,x_2) = u(x_1) + u(x_2) +\log(|x_1-x_2|) \leq 0,$$
for $\rho \otimes \rho$-almost every $(x_1,x_2)\in \R^d\times \R^d$.
Moreover, $F=0$ $\gamma$-almost everywhere. But then $F$ has a maximum on the support of $\gamma$ and so $\nabla F=0$ in this set; in particular we have that $\nabla u (x_1) = \frac{(x_1-x_2)}{|x_1-x_2|^2} $ on the support of $\gamma$. By solving this equation for $x_2$, we have
$$ x_2 = x_1 - \frac{\nabla u(x_1)}{\vert \nabla u(x_1)\vert^2}, \quad \gamma-\text{almost everywhere}, $$
which implies $\gamma = (Id,T)_{\sharp}\mu$ as we wanted to show.
\end{proof}

\section*{Acknowledgments} The authors thank Michael Seidl for fruitful discussions and for suggesting the charged wire model presented in the introduction of this paper.

\bibliographystyle{siam}
\bibliography{refsPhDThesis}

\end{document}